\documentclass[twoside,11pt]{amsart}

\usepackage{a4wide}
\usepackage{latexsym}
\usepackage{mathrsfs}
\usepackage[T1]{fontenc}
\usepackage{enumitem}
\usepackage{mathrsfs}

\usepackage{amssymb}
\usepackage{latexsym}
\usepackage{amsmath}
\usepackage{fancyhdr}
\usepackage{bbm}
\usepackage{setspace}
\usepackage{mathabx}

\numberwithin{equation}{section}

\newtheorem{theorem}{Theorem}[section] 
\newtheorem{lemma}[theorem]{Lemma}     

\newtheorem{proposition}[theorem]{Proposition}
\theoremstyle{definition}

\newtheorem{definition}[theorem]{Definition}

\newtheorem{remark}[theorem]{Remark}

\newcommand {\C}{\mathbb C}
\newcommand {\T}{\mathbb T}

\newcommand {\N}{\mathbb N}
\newcommand {\Z}{\mathbb Z}

\newcommand {\B}{\mathcal B}

\newcommand{\supp}{{\rm supp\,}}

\newcommand{\eps}{\varepsilon}
\newcommand{\spn}{{\rm span \,}}
\newcommand{\1}{\mathbbm{1}}

\newcommand{\id}{\operatorname{id}}

\newcommand{\ann}{\operatorname{ann}}



\begin{document}
	\title[Weak*-Simplicity of Convolution Algebras]{Weak*-Simplicity of Convolution Algebras on Discrete Groups}
	
		\author[J.\ T.\  White]{Jared T.\ White}
	\address{
		Jared T. White, School of Mathematics and Statistics, The Open University, Walton Hall, Milton Keynes MK7 6AA, United Kingdom}
	\email{jared.white@open.ac.uk}
	
	\date{November 2023}
	
	\maketitle
	
	\begin{center}
		\textit{The Open University}
	\end{center}
	
	\begin{abstract}
		We prove that, given a discrete group $G$, and $1 \leq p < \infty$, the algebra of $p$-convolution operators $CV_p(G)$ is weak*-simple, in the sense of having no non-trivial weak*-closed ideals, if and only if $G$ is an ICC group. This generalises the basic fact that $vN(G)$ is a factor if and only if $G$ is ICC. When $p=1$, $CV_p(G) = \ell^1(G)$. In this case we give a more detailed analysis of the weak*-closed ideals, showing that they can be described in terms of the weak*-closed ideals of $\ell^1(FC(G))$; when $FC(G)$ is finite, this leads to a classification of the weak*-closed ideals of $\ell^1(G)$.
	\end{abstract}
	
	\section{Introduction}
	Given a locally compact group $G$ and $1\leq p <\infty$ the algebra of $p$-convolution operators $CV_p(G)$ is defined to be the set of all bounded linear operators on $L^p(G)$ that commute with right translations by elements of $G$. When $G$ is discrete, $CV_p(G)$ may be identified with the set of those $\alpha \in \ell^p(G)$ with the property that $\xi \mapsto \alpha*\xi$ defines a bounded linear operator on $\ell^p(G)$. When $G$ is discrete and $p=1$ we recover the group algebra $\ell^1(G)$.
	The purpose of this article is to study the ideal structure of $CV_p(G)$ and $\ell^1(G)$, for a discrete group $G$. 
	
	
	Since $\ell^1(G)$ and $CV_p(G)$ are dual Banach algebras (in the sense of having a Banach space predual that is compatible with the algebra structure), it is natural to study their weak*-closed ideals. This continues a project of the author \cite{W3, W5} in which the weak*-closed left ideals of dual Banach algebras were studied, with particular focus on the measure algebra of a locally compact group, and the Banach algebra of operators on a reflexive Banach space.
	Our main theorem in this article is the following. 	We say that a dual Banach algebra is \textit{weak*-simple} if it has no proper, non-zero, weak*-closed ideals.
	
	\begin{theorem}		\label{thm1}
		Let $G$ be a group, and let $1 \leq p < \infty$. Then $CV_p(G)$ is weak*-simple if and only if $G$ is an ICC group.
	\end{theorem}
	
	\begin{proof}
		For $1<p<\infty$ apply Theorem \ref{3.8}; for $p=1$ apply Theorem \ref{iccthm}.
	\end{proof}
	
	 A von Neumann algebra is weak*-simple if and only if it is a factor. As such, Theorem \ref{thm1} can be seen as a non-trivial generalisation one of the most basic facts about the group von Neumman algebra $vN(G)$ of a discrete group $G$: namely that it is a factor if and only if $G$ is an ICC group. 
	
	When $p=1$ and $G$ is discrete, we have $CV_p(G) = \ell^1(G) = M(G)$, so that Theorem \ref{thm1} falls into the context of theorems about ideals of the measure algebra $M(G)$ of a locally compact group $G$. The structure of all norm-closed ideals of $M(G)$, for an infinite locally compact group $G$, is generally considered intractable, and even the set of maximal ideals can be very wild.
	For example, the maximal ideal space of $M(\T)$ contains uncountably many disjoint copies of the Stone--{\v C}ech compactification of $\Z$ \cite[Proposition 7]{OW14}; see also \cite{OWG16}.
	However, the author has found that if one restricts attention to the weak*-closed ideals, one can obtain some interesting descriptive results. In \cite{W3} the weak*-closed left ideals of $M(G)$, for a compact group $G$, were classified, and this was later extended to coamenable compact quantum groups by Anderson-Sackaney \cite{AS}. In \cite{W5} a description was obtained of the weak*-closed maximal left ideals of $M(G)$, for $G$ belonging to a certain class of groups including the connected nilpotent Lie groups. 
	
	In the present work we commence a study of the weak*-closed two-sided ideals of $\ell^1(G)$. Theorem \ref{thm1} with $p=1$ answers the most fundamental question in this line of enquiry.	
	However, we are able to give more detailed account of the weak*-closed ideals of $\ell^1(G)$ than for $CV_p(G) \ (1<p<\infty)$. Our next result says that the 
	weak*-closed ideals of $\ell^1(G)$ are determined by what happens on the $FC$-centre $FC(G)$.
	
	\begin{theorem}		\label{thm2}
		Let $\mathcal{T}$ be a transversal for $FC(G)$ in $G$. The weak*-closed ideals of $\ell^1(G)$ are given by
		$$\bigoplus_{t \in \mathcal{T}} \delta_t*J,$$
		where $J$ is a weak*-closed ideal of $\ell^1(FC(G))$ that is invariant for the action of $G$ by conjugation.
	\end{theorem}
	
	In the special case that $|FC(G)|<\infty$ this leads to the following classification theorem. Given $\Omega \subset \widehat{FC(G)}$ we write
	$$K(\Omega) : = \bigcap_{\pi \in \Omega} \ker \pi,$$
	which is an ideal of $\ell^1(FC(G)).$
	
	\begin{theorem}		\label{thm3}
		Suppose that $|FC(G)| < \infty$, and let $\mathcal{T}$ be a transversal for $FC(G)$ in $G$. The weak*-closed ideals of $\ell^1(G)$ are given by 
		$$\bigoplus_{t \in \mathcal{T}} \delta_t*K(\Omega),$$
		where $\Omega \subset \widehat{FC(G)}$ is a union of orbits for the action of $G$ on $\widehat{FC(G)}$ induced by conjugation.
	\end{theorem}
	
	Unfortunately, a classification theorem for the weak*-closed ideals of $\ell^1(G)$ for a general group $G$ seems beyond reach, as one runs into very hard questions about spectral synthesis, as we shall explain below in Remark \ref{abrem}.
	
	The paper in organised as follows. Section 2 establishes our notation and states some basic results that we shall use throughout the article; in Subsection 2.4 we discuss the proof of Theorem \ref{thm1} and explain why we cannot simply adapt the proof available for $vN(G)$. In Section 3 we prove that if $\mathcal{A}$ is any weak*-closed algebra lying between $PM_p(G)$ and $CV_p(G)$, where $1<p<\infty$, then $\mathcal{A}$ is weak*-simple if and only if $G$ is ICC; this establishes Theorem \ref{thm1} for $1<p<\infty$. The rest of the paper is devoted to studying $\ell^1(G)$. In Section 4 we define a uniform version of the ICC condition on a group, and show that it is actually equivalent to the ordinary ICC condition. Section 5 opens with a discussion of the weak*-closed ideals of the $\ell^1$-algebra of an abelian group. Next, we prove that $\ell^1(G)$ is weak*-simple if and only if $G$ is ICC, which completes Theorem \ref{thm1}. At the end of Section 5 we prove Theorem \ref{thm2} and Theorem \ref{thm3}.

	\section{Notation and Preliminary Remarks}
	
	\subsection{General Notation and Group Theory}
	For the remainder of the article all groups are assumed to be discrete, unless otherwise stated. 
	
	 Let $G$ be a group. Given $g \in G$ we write $ccl_G(g)$ for the conjugacy class of $g$ in $G$, and write $C_G(g)$ for the centraliser of $g$ in $G$. Given $g, t \in G$ we write $g^t= tgt^{-1}$.  Moreover, for  $1\leq p < \infty$,  $\alpha \in \ell^p(G)$, and $t \in G$, we define $\alpha^t =\delta_t*\alpha*\delta_{t^{-1}}$, so that 
	$$\alpha^t(g) = \alpha(t^{-1}gt) \qquad (g \in G).$$	
	Similarly, we write $T^t = \lambda_p(t)T\lambda_p(t^{-1}) \ (t \in G, \ T \in CV_p(G))$ (where $\lambda_p$ denotes the left regular representation of $G$ on $\ell^p(G)$).
	
	We say that $G$ is an \textit{ICC group} if every non-trivial conjugacy class is infinite. We define the \textit{$FC$-centre} of $G$ to be 
	$$FC(G) = \{ t \in G : |ccl_G(t)| < \infty \}.$$
	It turns out that $FC(G)$ is always a normal subgroup of $G$. Let 
	$$G_T = \{ g \in G : g^n =e \text{ for some } n \in \N \}$$
	be the set of torsion elements of $G$. By \cite[Theorem 1.6]{T} $FC(G)_T$ is a normal subgroup of $G$ containing the derived subgroup $FC(G)'$, so that $FC(G)$ is always a torsion-by-abelian group.	
	
	Let $H$ be a subgroup of $G$. We write $[G:H]$ for the index of $H$ in $G$. By a \textit{left (right) transversal} we mean a set of left (right) coset representatives for $H$ in $G$. 
	When $H$ is normal we simply talk about a transversal for $H$ in $G$. 
	We say that $H$ is an \textit{$FC$-central subgroup} of $G$ if $H \subset FC(G)$. 
	
	Given a group $G$ we write $\widehat{G}$ for the unitary dual of $G$, that is, the set of equivalence classes of irreducible, continuous, unitary representations of $G$ on Hilbert spaces.  Note that, for any $N \lhd G$, conjugation by elements of $G$ induces an action of $G$ on $\widehat{N}$ given by
	$$(t \cdot \pi)(n)= \pi(t^{-1}nt) \quad (t \in G, \ n \in N, \  \pi \in \widehat{N}).$$	
	
	We now fix some set theory notation. Given sets $B \subset A$ we write $ \id_A$ for the identity map on $A$, and $\1_B \colon A \to \{ 0,1\}$ for the characteristic function of $B$. We write $B \subset \subset A$ to mean that $B$ is a finite subset of $A$. 
	
	Given $1 \leq p \leq \infty$ and $\xi \in \ell^p(A)$ we write $\xi|_B \in \ell^p(B)$ for the restriction of $\xi$ to $B$. In an abuse of notation we often regard $\ell^p(B)$ as a subset of $\ell^p(A)$ is the obvious way. Moreover we often identify $\xi|_B$ with its image inside $\ell^p(A)$, so that as a function on $A$ we have $\xi|_B(x) = \xi(x)$ if $x \in B$ and $\xi|_B(x) = 0$ otherwise.
	

	\subsection{Dual Banach Algebras}
	
	By a \textit{dual Banach algebra} we mean a pair $(\mathcal{A},\mathcal{A}_*)$, where $\mathcal{A}$ is a Banach algebra, 
	and $\mathcal{A}_*$ is a Banach space with $(\mathcal{A_*})^* = \mathcal{A}$, with the property that the multiplication on $\mathcal{A}$ is separately weak*-continuous. Examples include the measure algebra $M(G)$ of a locally compact group $G$, with predual $C_0(G)$, and hence in particular $\ell^1(G)$ for a discrete group $G$. Von Neumann algebras can be characterised as those C*-algebras that are dual Banach algebras. 
	
	Let $E$ and $F$ be Banach spaces. We write $\B(E,F)$ for the set of bounded linear operators from $E$ to $F$, and write $\B(E) := \B(E,E)$. The Banach space $\B(E,F^*)$ may be identified with $(F \widehat{\otimes} E)^*$, where $\widehat{\otimes}$ denotes the projective tensor product of Banach spaces, with the duality given by
	\begin{equation}		\label{eq2.2}
		\langle \sum_{j=1}^\infty y_j \otimes x_j, T \rangle = \sum_{j=1}^\infty \langle y_j, Tx_j \rangle \qquad (T \in \B(E,F^*), \ \sum_{j=1}^\infty y_j \otimes x_j \in F \widehat{\otimes} E ). 
	\end{equation}
	When $E$ is a reflexive Banach space, identifying $E$ with $E^{**}$ means that $\B(E)$ can be identified with $(E^* \widehat{\otimes} E)^*$, and this makes $\B(E)$ into a dual Banach algebra.
	
	Let $X$ be a subspace of $E$ and $Y$ a subspace of $E^*$. We define
	$$X^\perp = \{ \varphi \in E^* : \langle x, \varphi \rangle = 0, \ x \in X \} \quad \text{and} \quad
	Y_\perp = \{ x \in E : \langle x, \varphi \rangle = 0, \ \varphi \in Y \}.$$
	Note that $X^\perp$ is always weak*-closed and that $(Y_\perp)^\perp = \overline{Y}^{w^*}$. 
	It follows that a linear subspace of $E^*$ is weak*-closed if and only if it has the form $X^\perp$, for some $X \leq E$.
	
	A dual Banach algebra is \textit{weak*-simple} if it has no proper, non-zero, weak*-closed ideals. For example, for any reflexive Banach space $E$ the dual Banach algebra $\B(E)$ is weak*-simple since the ideal of finite rank operators is contained in every non-zero ideal and is weak*-dense. A von Neumann algebra is weak*-simple if and only if it is a factor. To the best of our knowledge, the term `weak*-simple' is new, although it is an obvious notion.
	
	Let $(\mathcal{A}, \mathcal{A}_*)$ be a dual Banach algebra, and let $E$ be a Banach left $\mathcal{A}$-module. Then $E^*$ becomes a Banach right $\mathcal{A}$-module via
	$$\langle x, \varphi \cdot a \rangle = \langle a \cdot x, \varphi \rangle \quad (x \in E, \ a \in \mathcal{A}, \ \varphi \in E^*),$$
	and we say that $E^*$ is a \textit{dual right $\mathcal{A}$-module}. We say that $E^*$ is \textit{normal} if the map $\mathcal{A} \to E^*$ given by
	$a \mapsto \varphi \cdot a$ is weak*-continuous for each $\varphi \in E^*$. Given a Banach right $\mathcal{A}$-module $F$ we define its annihilator to be
	$$\ann(F) : = \{ a \in \mathcal{A} : x \cdot a = 0 \ (x \in F)\}.$$
	When $F$ is a normal, dual right $\mathcal{A}$-module, $\ann(F)$ is a weak*-closed, two-sided ideal of $\mathcal{A}$. 
	Given Banach left $\mathcal{A}$-modules $E_1$ and $E_2$ we define $\operatorname{Hom}^l_{\mathcal{A}}(E_1, E_2)$ to be the set of bounded left $\mathcal{A}$-module homomorphisms from $E_1$ to $E_2$. As we shall detail in Section 3, $\operatorname{Hom}^l_{\mathcal{A}}(E_1,E_2)$ is itself a Banach right $\mathcal{A}$-module via the formula $(\Phi\cdot a)(x) = \Phi(x)\cdot a$, where $a \in \mathcal{A}, \ x \in E_1,$ and $\Phi \in \operatorname{Hom}^l_{\mathcal{A}}(E_1, E_2)$.
	

	\subsection{Convolution Operators on Groups}
		
	Banach algebras of $p$-convolution operators on groups have a long history of study going back at least to the work of Herz \cite{H71, H73} and Fig{\`a}-Talamanca \cite{FT65}. 
	They have also been the subject of serious study by a number of mathematicians since - see e.g. \cite{Co, De11, DFM, DG, GT22}. In recent years, the study of Banach algebras of $p$-convolution operators on groups has seen renewed interest in light of the program, initiated primarily by Phillips \cite{P12, P13}, to study so-called $p$-operator algebras; see \cite{G21} for a survey of this field. We now give the formal definitions of these algebras.
	
	Let $G$ be a group. Given $1\leq p < \infty$, we write $\lambda_p$ and $\rho_p$ for the left and right regular representations of $G$ on $\ell^p(G)$, respectively. We define the algebra of $p$-convolution operators on $G$ to be
	$$CV_p(G) = \rho_p(G)' = \{ T \in \B(\ell^p(G)) : T \rho_p(t) = \rho_p(t) T \ (t \in G) \}.$$
	In fact $CV_p(G) = \lambda_p(G)''$, and when $p=1$ we have $CV_p(G) = \ell^1(G)$. For $1<p<\infty$ we define the algebra of $p$-pseudomeasures on $G$ to be
	$$PM_p(G) = \overline{\spn}^{w^*} \lambda_p(G),$$
	where $w^*$ refers to the weak*-topology on $\B(\ell^p(G))$. 
	
	We always have $\lambda_p(\ell^1(G)) \subset PM_p(G) \subset CV_p(G)$.
	When $G$ has the approximation property, we have $PM_p(G) = CV_p(G)$ for all $1 < p<\infty$; this observation seems to be folklore, but a proof is given in \cite{DS}. Whether one always has $PM_p(G) = CV_p(G)$ for all $p$ and all $G$ remains a difficult unsolved problem. Our main theorem for $p \neq 1$, Theorem \ref{3.8}, is stated for any weak*-closed Banach algebra sitting between $PM_p(G)$ and $CV_p(G)$, so that our results do not depend the answer to this open question. 
	
	For $1<p<\infty$ the algebra $CV_p(G)$ is a dual Banach algebra, whose predual is denoted by $\overline{A_p}(G)$,
	a concrete description of which was given by Cowling \cite{Co}. Similarly, $PM_p(G)$ is a dual Banach algebra with predual $A_p(G)$, the \textit{Fig{\`a}-Talamanca--Herz algebra}. See \cite{De11} for background on $A_p(G)$. The weak*-topology on $CV_p(G)$ and $PM_p(G)$ corresponding to these preduals coincides with that given by considering them as subalgebras of $\B(\ell^p(G))$ with predual $\ell^{p'}(G) \widehat{\otimes} \ \ell^{p}(G)$. As such, by using \eqref{eq2.2}, we shall not need the precise definitions of $A_p(G)$ or $\overline{A_p}(G)$ in this article.
	
	We write $A(G) = A_2(G)$ for the Fourier algebra of $G$. We shall make frequent use of the basic properties of $A(G)$ for $G$ a compact abelian group, for which \cite{Ru} is a good reference. 
	
	In the next proposition we summarise some basic properties of $p$-convolution operators on discrete groups, for $p$ in the reflexive range, that we shall use throughout Section 3. We have been unable to find a reference, but the proof is a direct generalisation of the case $p=2$, which is given in \cite[\S 4.25]{SZ}.
	
	\begin{proposition}		\label{2.1}
		Let $G$ be a group, and let $1<p< \infty$. 
		\begin{enumerate}
			\item[(i)] Given $T \in CV_p(G)$ we have 
			$$T\xi = \alpha*\xi \quad (\xi \in \ell^p(G)),$$
			where $\alpha = T\delta_e$. 
			\item[(ii)] Moreover, we have 
			$$\langle \delta_t, T \delta_s \rangle = \alpha(ts^{-1}) \quad (s,t \in G) .$$
			\item[(iii)] Given $S,T \in CV_p(G)$, we have $S=T$ if and only if $S\delta_e = T\delta_e$; that is, $\delta_e$ is a separating vector for the action of $CV_p(G)$ on $\ell^p(G)$.
		\end{enumerate}
	\end{proposition}
	
	\noindent
	Using this proposition, we can characterise the centre of $CV_p(G)$.
	
	\begin{lemma}		\label{0.3}
		Let $G$ be a group. Then 
		$$Z(CV_p(G)) = \{ T \in CV_p(G) : T\delta_e \text{ is constant on conjugacy classes}\},$$
		and in particular $CV_p(G)$ has trivial centre if and only if $G$ is ICC.
	\end{lemma}
	
	\begin{proof}
		Let $T \in Z(CV_p(G))$. Then given $t \in G$ we have $\lambda_p(t)T\lambda_p(t^{-1}) = T$, so that
		$$T\delta_e = \lambda_p(t)T\lambda_p(t^{-1})  \delta_e = \delta_t*(T\delta_e)*\delta_{t^{-1}}.$$
		Hence $T\delta_e$ is constant on conjugacy classes.
		
		Now suppose that $T \in CV_p(G)$ and that $\alpha :=T\delta_e$ is constant on conjugacy classes. Let $S \in CV_p(G)$, and let $\beta = S\delta_e$. Given $g \in G$ we have
		\begin{align*}
			\langle \delta_g, TS\delta_e \rangle &= (\alpha*\beta)(g) = \sum_{s \in G} \alpha(gs^{-1})\beta(s) = \sum_{s \in G} \alpha(s^{-1}(gs^{-1})s)\beta(s) \\
			&= \sum_{s \in G} \beta(s) \alpha(s^{-1}g) = (\beta*\alpha)(g) = \langle \delta_g, ST \delta_e \rangle.
		\end{align*}
		It follows that $TS=ST$, and so $T$ is central.
	\end{proof}
	
	Let $H$ be a subgroup of $G$. Then we can identify $CV_p(H)$ with a subalgebra of $CV_p(G)$ as follows. First, fix a right transversal $\mathcal{T}$ for $H$ in $G$, and define
	$\xi_t \in \ell^p(H)$ by $\xi_t(u) = \xi(ut) \ (u \in H, t \in \mathcal{T}).$
	Then $S \in CV_p(H)$ can be extended to an operator on $\ell^p(G)$ via
	\begin{equation}		\label{eq2.1}
	S\xi = \sum_{t \in \mathcal{T}} (S\xi_t)*\delta_t \qquad (S \in CV_p(H), \ \xi \in \ell^p(G)),
	\end{equation}
	and it can be checked that this extension belongs to $CV_p(G)$. The proof of this result is routine in the discrete case; see \cite[Section 7.1, Theorem 13]{De11} for a generalisation that holds for locally compact $G$.	
	
	Furthermore, given an operator $T \in CV_p(G)$ we may define its \textit{restriction to $H$}, written $T_H \in CV_p(H)$, by
	$$T_H \xi = P_H T \xi^G \qquad ( \xi \in \ell^p(H)),$$
	where $\xi^G$ denotes the obvious extension of $\xi$ to $G$ by  adding zeros, and $P_H \colon \ell^p(G) \to \ell^p(H)$ is the canonical projection. We can use this to give a right coset decomposition for $T$, by setting 
	$$T_t = (T \lambda_p(t^{-1}) )_H \qquad (t \in \mathcal{T}).$$
	Writing $\alpha = T\delta_e$, an easy calculation then shows that
	\begin{equation*}	
		T_t\delta_e = \alpha_t \qquad (t \in \mathcal{T}).
	\end{equation*}
	Unfortunately, $\sum_{t \in \mathcal{T}} T_t\lambda_p(t)$ does not usually converge to $T$ in the weak*-topology,
	 but we do have $\alpha = \sum_{t \in \mathcal{T}} \alpha_t*\delta_t$ convergent in $\ell^p$-norm.

	\subsection{Remarks on the Proof of Theorem \ref{thm1}}
	
	We shall break the proof of Theorem \ref{thm1} into two parts: the case $1<p<\infty$ is given by Theorem \ref{3.8}, whereas the case $p=1$ is given by Theorem \ref{iccthm}. Theorem \ref{3.8} actually states the analogous result for any weak*-closed subalgebra $\mathcal{A}$ of $CV_p(G)$ containing $PM_p(G)$.
	
	When $p=2$, we have $CV_p(G) = vN(G)$. For von Neumann algebras, weak*-simplicity is equivalent to being a factor. Indeed, by standard results a non-trivial weak*-closed ideal of a von Neumann algebra is generated by a non-trivial central idempotent; moreover whenever the centre of the von Neumann algebra is non-trivial it contains a non-trivial idempotent, and this generates a weak*-closed ideal. Since (as is well known) $vN(G)$ is a factor if and only if $G$ is ICC, Theorem \ref{thm1} follows in this special case.
	
	More generally, Lemma \ref{0.3} shows that $CV_p(G)$, for $1 \leq p < \infty$, has trivial centre if and only if $G$ is ICC; however, for $p \neq 2$ the connection between weak*-closed ideals and the centre of the algebra is less clear. Moreover, $CV_p(G)$ will not contain as many idempotents as $vN(G)$ when $p \neq 2$, so we should not expect every weak*-closed ideal to be generated by an idempotent. As an example, $\ell^1(\Z)$ contains many weak*-closed ideals (See Lemma \ref{yc} and Remark \ref{abrem}), but contains no non-trivial idempotents (by e.g. \cite[\S 3.2.1]{Ru}). 
	
	In the following proposition we observe that, in contrast to von Neumann algebras, a general dual Banach algebra can have a non-trivial weak*-closed ideal that intersects the centre only at $\{0\}$. By a \textit{weight} on a group $G$ we mean a function $\omega \colon G \to [1, \infty)$ 
	such that $\omega(e) = 1$ and $\omega(st) \leq \omega(s)\omega(t) \ (s,t \in G)$. Note that $\ell^1(G, \omega)$ is then a dual Banach algebra with predual $c_0(G, 1/\omega)$ (see e.g. \cite[Proposition 5.1]{D06}). 
	
	\begin{proposition}		\label{2.2}
		Let $G$ be a finitely-generated ICC group, fix a finite generating set, and denote corresponding the word-length of $t \in G$ by $|t|$.
		Let $c>1$ and let $\omega \colon G \to [1, \infty)$ be given by $\omega(t) = c^{|t|}$.
		Then the augmentation ideal 
		$$\ell_0^1(G, \omega) := \left\{ f \in \ell^1(G, \omega) : \sum_{t \in G} f(t) = 0 \right\}$$
		 is weak*-closed, but intersects the centre trivially.
		In particular $Z(\ell^1(G, \omega)) = \C \delta_e$, 
		but $\ell^1(G, \omega)$ is not weak*-simple. 
	\end{proposition}

	\begin{proof}
		Due to the choice of $\omega$ the constant functions $\C1$ are contained in $c_0(G, 1/\omega)$, and $\ell_0^1(G, \omega) = (\C 1)^\perp$.
		As such $\ell_0^1(G, \omega)$ is weak*-closed, and of course it is an ideal.
		Any element of the centre of $\ell^1(G,\omega)$ must be constant on conjugacy classes, and since it also belongs to $c_0(G)$, must be a multiple of $\delta_e$.
		Therefore $Z(\ell^1(G, \omega)) = \C \delta_e$, and clearly $\ell_0^1(G, \omega) \cap \C\delta_e = \{ 0 \}$.
	\end{proof}
	
	In Proposition \ref{3.1}(i) we shall show that, for $p$ in the reflexive range, every weak*-closed ideal of $CV_p(G)$ does contain a non-zero central element; we have not been able to see how to generalise this to $\ell^1(G)$, and our proof of Theorem \ref{thm1} in that case takes a different path. However, for $\ell^1(G)$ there are other nice tools available. In particular $\spn \{ \delta_t : t \in G \}$ is norm dense, whereas for $CV_p(G)$, where $p \neq 1,2$, we don't know in general whether $\spn \{ \lambda_p(t) : t \in G \}$ is even weak*-dense.
	

	\section{Weak*-Simplicity of $CV_p(G)$ and $PM_p(G)$ for $1<p<\infty$}
	
	Part (ii) of the following proposition is needed for our proof of Theorem \ref{thm1}. Part (i) may be of independent interest.
	It can be seen as an analogue of the following property of von Neumann algebras: 
	every non-zero weak*-closed ideal contains a non-zero central projection.
	
	\begin{proposition}		\label{3.1}
		Let $G$ be a group, let $1<p<\infty$, and let $\mathcal{A}$ be a weak*-closed subalgebra of $\B(\ell^p(G))$ with 
		$$PM_p(G) \subset \mathcal{A} \subset CV_p(G).$$
		Then 
		\begin{enumerate}
			\item[\rm (i)] every non-zero weak*-closed ideal of $\mathcal{A}$ has non-zero intersection with $Z(\mathcal{A})$;
			\item[\rm (ii)] if $G$ is ICC then $\mathcal{A}$ is weak*-simple.
		\end{enumerate}
	\end{proposition}
	
	\begin{proof}
		We shall prove (i) and (ii) together. Let $\{ 0 \} \neq I \lhd \mathcal{A}$ be weak*-closed. Let $T \in I \setminus \{ 0 \}$. By translating and scaling $T$ we may assume that $\langle \delta_e, T \delta_e \rangle = 1$. Let
		$$X = \{ S \in I : \langle \delta_e, S\delta_e \rangle = 1, \ \|S \| \leq \|T \| \},$$
		which is a non-empty, weak*-compact, convex subset of $\B(\ell^p(G))$.
		
		Let $\psi \colon \B(\ell^p(G)) \to \ell^p(G)$ be given by 
		$$\psi(S) = S \delta_e \quad (S \in \B(\ell^p(G))).$$
		Observe that $\psi$ is weak*-weakly continuous: indeed, if $\eta \in \ell^{p'}(G)$ then, for all $S \in \B(\ell^p(G))$, we have
		$\langle \psi(S), \eta \rangle = \langle S\delta_e, \eta \rangle 
		= \langle \eta \otimes \delta_e, S \rangle.$
		As such $Y : = \psi(X)$ is a weakly-compact convex subset of $\ell^p(G)$, as it is the continuous linear image of a weak*-compact convex set.
		
		Let $S \in \mathcal{A}$, and let $\alpha = S\delta_e = \psi(S).$
		The action of $G$ on $Y$ by conjugation is well-defined: indeed,
		given $t, g \in G$ we have
		$$\langle \delta_t, S^g \delta_e \rangle = \langle \delta_{g^{-1}t}, S\delta_g \rangle 
		= \langle \delta_{g^{-1}tg}, S\delta_e \rangle = \alpha^g(t),$$
		so that $\psi(S)^g = \psi(S^g) \in Y$. Moreover, the conjugation action of $G$ on $\ell^p(G)$ is isometric.
		
		By the previous paragraph, we my apply the Ryll-Nardzewski Fixed Point Theorem to the conjugation action of $G$ on $Y$ to find $y_0 \in Y$ such that $y_0^g = y_0$ for all $g \in G$. Say $y_0 = \psi(T_0)$, for $T_0 \in X$. 
		By Lemma \ref{0.3}, since $y_0 = T_0 \delta_e$ is constant on conjugacy classes, we must have $T_0 \in Z(\mathcal{A})$, and $T_0$ belongs to $I \setminus \{0 \}$ since it is in $X$, and this proves (i). If $G$ is ICC, then the fact that $y_0 \in \ell^p(G)$ is non-zero and constant on conjugacy classes forces $y_0 = \delta_e$, and hence $T_0 = \id_{\ell^p(G)}$ and $I = \mathcal{A}$. As $I$ was arbitrary, $\mathcal{A}$ must be weak*-simple, and we have proved (ii).
	\end{proof}

		We shall spend the remainder of Section 3 constructing a non-trivial weak*-closed ideal of $CV_p(G)$ when $FC(G) \neq \{ e\}$. Note that if $PM_p(G) = CV_p(G)$ this task becomes much easier, as we can adapt the approach that we shall use later for $\ell^1(G)$ (see Case 2 in the proof of Theorem \ref{iccthm}). The difficulty of the present situation comes from the fact that, if we don't know that $\lambda_p(G)$ has weak*-dense linear span in $CV_p(G)$, then we cannot check that a weak*-closed subset $I \subset CV_p(G)$ is an ideal just by checking that $\lambda_p(t)T, T \lambda_p(t) \in I$ for all $t \in G$ and $T \in I$. Instead, our strategy is to take the annihilator of a well-chosen normal, dual right $CV_p(G)$-module, as this will automatically be a weak*-closed ideal. A careful choice is required to ensure that the ideal is non-zero and proper.

	\begin{lemma}		\label{3.4}
		Let $(\mathcal{A}, \mathcal{A}_*)$ be a dual Banach algebra,  let $\mathcal{S}$ be a closed subalgebra of $\mathcal{A}$, and let $E$ be a Banach left $\mathcal{S}$-module.
		Then $\operatorname{Hom}^l_{\mathcal{S}}(E, \mathcal{A})$ is a normal, dual, right $\mathcal{A}$-module, with the action given by
		\begin{equation}		\label{eq4}
			(\Phi \cdot a)(x) = \Phi(x)a \quad (\Phi \in \operatorname{Hom}^l_{\mathcal{S}}(E, \mathcal{A}), \ a \in \mathcal{A}, \ x \in E).
		\end{equation}
	\end{lemma}
	
	\begin{proof}
		Consider $\mathcal{A}_* \widehat{\otimes} E$, whose dual space may be identified with $\B(E, \mathcal{A})$ via \eqref{eq2.2}, and define
		$$Z = \overline{\spn}\{ (u \cdot b) \otimes x - u\otimes (b \cdot x) : u \in \mathcal{A}_*, \ x \in E, \ b \in \mathcal{S} \}.$$
		Then we have
		$$( \mathcal{A}_* \widehat{\otimes} E / Z )^* \cong Z^\perp,$$
		and a routine calculation shows that, as a subset of $\B(E, \mathcal{A})$,   
		$$Z^\perp = \operatorname{Hom}^l_{\mathcal{S}}(E, \mathcal{A}).$$
		As such $\operatorname{Hom}^l_{\mathcal{S}}(E, \mathcal{A})$ is a dual space.
		
		Now, $\mathcal{A}$ acts on $\mathcal{A}_* \widehat{\otimes} E$ on the left by 
		$$a \cdot \left( \sum_{j=1}^\infty u_j \otimes x_j \right) = \sum_{j = 1}^\infty (a \cdot u_j)\otimes x_j 
		\quad (a \in \mathcal{A}, \ u_j \in \mathcal{A}_*, \ x_j \in E),$$
		and $Z$ is a closed submodule for this action, so that $\mathcal{A}_* \widehat{\otimes} E / Z$ becomes a Banach left $\mathcal{A}$-module. A routine calculation shows that the dual action is exactly the action we defined on $\operatorname{Hom}^l_{\mathcal{S}}(E, \mathcal{A})$ by \eqref{eq4}.
		
		Finally, we show that $\operatorname{Hom}^l_{\mathcal{S}}(E, \mathcal{A})$ is normal. 
		Let $\Phi \in \operatorname{Hom}^l_{\mathcal{S}}(E, \mathcal{A})$, and define 
		$L_\Phi \colon \mathcal{A} \to \operatorname{Hom}^l_{\mathcal{S}}(E, \mathcal{A})$ by  $L_\Phi \colon a \mapsto \Phi \cdot a$. 
		We want to show that $L_\Phi$ is weak*-continuous, so take arbitrary $\tau \in \operatorname{Hom}^l_{\mathcal{S}}(E, \mathcal{A})_*$ and consider $\tau \circ L_\Phi$. 
		We can write $\tau$ as $\sum_{j=1}^\infty u_j \otimes x_j +Z$, 
		where $u_j \in \mathcal{A}_*$, $x_j \in E$, and $\sum_{j=1}^\infty \|u_j\| \|x_j \| < \infty.$ Given $a \in \mathcal{A}$, we have
		\begin{align*}
			(\tau \circ L_\Phi)(a) &= \langle \sum_{j=1}^\infty u_j \otimes x_j, \Phi \cdot a \rangle 
			= \sum_{j=1}^\infty \langle u_j, (\Phi \cdot a)(x_j) \rangle \\
			&= \sum_{j=1}^\infty \langle u_j, \Phi(x_j)a \rangle 
			= \sum_{j=1}^\infty \langle u_j \cdot \Phi(x_j), a \rangle.
		\end{align*}
		Note that $\sum_{j = 1}^\infty \| u_j \cdot \Phi(x_j) \| \leq \| \Phi \| \sum_{j=1}^\infty \|u_j\| \|x_j \| < \infty$, so that 
		$\sum_{j = 1}^\infty u_j \cdot  \Phi(x_j)$ converges to an element of $\mathcal{A}_*$ which we denote $u_0$. The above calculation shows that $(\tau \circ L_\Phi)(a) = \langle u_0, a \rangle$. Since $\tau$ was arbitrary, this proves that $L_\Phi$ is weak*-continuous.
	\end{proof}
	
	\begin{lemma}		\label{3.6}
		Let $G$ be a group, and let $N$ be an abelian, normal subgroup of $G$. Let $K$ be a Banach left $CV_p(N)$-module, and let $J$ be an ideal of $CV_p(N)$ that annihilates $K$, and that is invariant for the conjugation action of $G$ on $CV_p(N)$. 
		Write $M = \operatorname{Hom}^l_{CV_p(N)}(K, CV_p(G))$, and define 
		$$I := \ann(M) =  \{ T \in CV_p(G) : \Phi \cdot T = 0, \ \Phi \in \operatorname{Hom}^l_{CV_p(N)}(K, CV_p(G)) \}.$$
		Then $I$ is a weak*-closed ideal of $CV_p(G)$ containing $J$.
	\end{lemma}

	\begin{proof}
		Since (by Lemma \ref{3.4}) $I$ is the annihilator of a normal, dual, right $CV_p(G)$-module, it is easily seen to be a weak*-closed two-sided ideal. It remains to show that it contains $J$. 
				
		Let $\Phi \in M$. We shall show that $\Phi \cdot S = 0$ for every $S \in J$. 
		Define $\alpha \colon K \to \ell^p(G)$ by
		$$\alpha(v)(g) = \langle \delta_g, \Phi(v) \delta_e \rangle \quad (v \in K, \ g \in G).$$
		Fix a transversal $\mathcal{T}$ for $N$ in $G$, and for each $t \in \mathcal{T}$ define functions $\Phi_t \colon K \to CV_p(N)$ and $\alpha_t \colon K \to \ell^p(G)$  by 
		$$\Phi_t(v) = [\Phi(v)]_t \qquad ( t \in  \mathcal{T}, \ v \in K),$$ 
		and
		$$\alpha_t(v)(g) = \alpha(v)(gt) \ (v \in K, \ g \in H), \quad \text{ and } \quad \alpha_t(v)(g) = 0 \ (v \in K, \ g \in G \setminus H),$$
		so that $\Phi_t(v) \delta_e = \alpha_t(v) \ (v \in K)$. 
		
		We next show that each map $\Phi_t$ is a left $CV_p(N)$-module homomorphism. Indeed, given $S \in CV_p(N)$ we have
		\begin{align*}
			S\Phi(v)\delta_e =  S \sum_{t \in \mathcal{T}} \alpha(v)|_{Nt} 
			= S \sum_{t \in \mathcal{T}} \alpha_t(v)*\delta_t 
			= \sum_{t \in \mathcal{T}} (S\alpha_t(v))*\delta_t
			\quad \text{(by \eqref{eq2.1}}). 
		\end{align*}
		Also
		\begin{align*}
			S \Phi(v) \delta_e &= \Phi(S \cdot v) \delta_e 
			= \sum_{t \in \mathcal{T}} \alpha(S \cdot v)|_{Nt} 
			=  \sum_{t \in \mathcal{T}} \alpha_t(S \cdot v)*\delta_t.
		\end{align*}
		By comparing these two expressions, we see that we must have $S\alpha_t(v) = \alpha_t(S \cdot v)$, for each $v \in K$ and $t \in \mathcal{T}$, and hence also $S\Phi_t(v) = \Phi_t(S \cdot v)$, since $\delta_e$ is a separating vector. 
		
		Now, fix $S \in CV_p(N)$, and let $\beta(n) = \langle \delta_n, S\delta_e \rangle \ (n \in N)$, so that 
		$$S \xi = \beta*\xi \quad (\xi \in \ell^p(N)).$$
		We claim that 
		\begin{equation}		\label{eq3.1}
			\sum_{t \in \mathcal{T}} \beta^t*\alpha_t(v) *\delta_t = \alpha(v)*\beta \quad (v \in K).
		\end{equation}
		Note that the left hand side converges because the terms $\beta^t*\alpha_t(v) *\delta_t$ are disjointly supported, and 
		$$ \sum_{t \in \mathcal{T}} \| \beta^t*\alpha_t(v) *\delta_t \|_p^p
		= \sum_{t \in \mathcal{T}} \| S^t(\alpha_t(v) *\delta_t )\|_p^p
		\leq \| S \|^p \sum_{t \in \mathcal{T}} \| \alpha_t(v) \|_p^p = \|S\|^p \|\alpha(v) \|_p^p < \infty.$$
		We suppress $v$ for the time being.
		Let $n \in N, \ s \in \mathcal{T}$, and consider $g = ns \in Ns$. We have
		\begin{align*}
			\left( \sum_{t \in \mathcal{T}} \beta^t*\alpha_t *\delta_t \right)(g) &= \left( \beta^s*\alpha_s *\delta_s \right)(ns) 
			= \sum_{u \in N} \beta^s(u)(\alpha_s*\delta_s)(u^{-1}ns) \\
			&=  \sum_{u \in N} \beta^s(u)\alpha(u^{-1}ns).
		\end{align*}
		 We also have
		\begin{align*}
			(\alpha*\beta)(g) &= \sum_{u \in N} \alpha(nsu^{-1})\beta(u) 
			= \sum_{u \in N} \beta(s^{-1}us)\alpha(ns(s^{-1}us)^{-1}) \\ 
			&= \sum_{u \in N} \beta(s^{-1}us)\alpha(nss^{-1}u^{-1}s) = \sum_{u \in N} \beta(s^{-1}us)\alpha(nu^{-1}s) \\
			&= \sum_{u \in N} \beta^s(u)\alpha(u^{-1}ns), 
		\end{align*}
		which proves the claim.
		
		Now suppose that $S \in J$. We shall show that $\Phi \cdot S = 0$. Let $v \in K$. Then
		\begin{align*}
			(\Phi \cdot S)(v)\delta_e &= \Phi(v)S\delta_e = \alpha(v)*\beta  
			= \sum_{t \in \mathcal{T}}\beta^t*\alpha_t(v)*\delta_t \quad (\text{by \eqref{eq3.1}}) \\
			&= \sum_{t \in \mathcal{T}} S^t\Phi_t(v) \delta_t \quad (\text{by \eqref{eq2.1}}) \\
			&= \sum_{t \in \mathcal{T}} \Phi_t(S^t \cdot v)\lambda(t) \delta_e .  
		\end{align*}
		Since $J$ is $G$-conjugation-invariant, $S^t \in J$ for each $t \in \mathcal{T}$, and so each $S^t \cdot v = 0$, and the above sum is $0$. Since $\delta_e$ is a separating vector, this forces $\Phi \cdot S = 0$, as required.
	\end{proof}
	
	
	
	\begin{lemma}		\label{3.7}
		Let $G$ be a group, and let $N$ be a normal $FC$-central subgroup of $G$, isomorphic to $\Z^n$, for some $n \in \N$. Then $CV_p(G)$ has a proper, weak*-closed ideal $I$ such that $I \cap \lambda_p(\ell^1(G)) \neq \{ 0 \}$.
	\end{lemma}
	
	\begin{proof}
		We identify $\widehat{N}$ with $\T^n$. Let $U_0 = \{ \boldsymbol{z} \in \T^n : |z_i - 1|<1, \ i = 1, \ldots, n \},$ which is an open subset of $\T^n$ containing the identity. 
		Let $U = \bigcap_{t \in G} U_0^t $. Since $N$ is finitely generated and $FC$-central, there are only finitely many distinct automorphisms of $N$ 
		induced by conjugation by elements of $G$, and as such the intersection defining $U$ is actually a finite intersection. 
		It follows that $U$ is an open set, and it is non-empty as $e \in U$. Now, let $E = \overline{U}$. 
		Then $E$ is a compact neighbourhood of the identity in $\T^n$ with $0<m(E)<1$, which is invariant for the action of $G$ induced by conjugation.
		
		Define
		$$J = \{ T \in CV_p(N) : \widehat{T}|_E = 0\}, \qquad K = \{ T \in CV_p(N) : \widehat{T}|_{\widehat{N} \setminus E} = 0\}.$$
		Then $J$ and $K$ are ideals in $CV_p(N)$, with $J = \ann(K)$. 
		Because $E$ is $G$-conjugation-invariant,  so is $J$. 
		Since the Fourier algebra is regular, there exists $u  \in A(\widehat{N}) \setminus \{ 0 \}$ such that $u|_E = 0$, and taking the inverse Fourier transform of $u$ gives a non-zero element $f$ of $\ell^1(G)$ such that $\lambda_p(f)$ belongs to $J$. Similarly $K \neq \{ 0\}.$
				
		Let 
		$$I = \ann \operatorname{Hom}^l_{CV_p(N)}(K, CV_p(G)).$$
		By Lemma \ref{3.6} $I$ is weak*-closed, and contains $J$, so that $I \cap \lambda_p(\ell^1(G)) \neq \{ 0 \}$.		
		Finally, we must show that $I$ is proper. Note that composing the inclusion maps $K \to CV_p(N) \to CV_p(G)$ gives 
		a non-zero left $CV_p(N)$-module homomorphism $K \to CV_p(G)$, so that $\operatorname{Hom}^l_{CV_p(N)}(K, CV_p(G))$ is non-zero.
		As such $\id_{\ell^p(G)}$ is an element of $CV_p(G)$ not belonging to $I$. This completes the proof.
	\end{proof}
	
	\begin{theorem}		\label{3.8}
		Let $G$ be a group, let $1<p<\infty$, and let $\mathcal{A}$ be a weak*-closed subalgebra of $\B(\ell^p(G))$ with 
		$$PM_p(G) \subset \mathcal{A} \subset CV_p(G).$$
		Then $\mathcal{A}$ is weak*-simple if and only if $G$ is ICC.
	\end{theorem}
	
	\begin{proof}
		If $G$ is ICC then $\mathcal{A}$ is weak*-simple by Proposition \ref{3.1}(ii). We shall prove the converse.
		
		Suppose that $G$ is not ICC. Then $FC(G) \neq \{e\}$ and we shall use the structure of $FC(G)$ to build a weak*-closed ideal of $\mathcal{A}$.
		 There are two possibilities concerning $FC(G)_T$.
		
		\vskip 2mm
		\underline{Case 1: $FC(G)_T \neq \{e \}.$} In this case, $G$ contains a finite normal subgroup $N$ by Dicman's Lemma  \cite[Lemma 1.3]{T}. We suppose that $G \neq N$, since otherwise the result is trivial. Let $P = \lambda_p(\frac{1}{|N|}\1_N)$. Then $P$ is a central idempotent in $\mathcal{A}$ and $\mathcal{A}P$ is a weak*-closed ideal. Note that $\mathcal{A}P \neq \{0\}$.  Moreover, $\mathcal{A}P \neq \mathcal{A}$ because 
		$\id_{\ell^p(G)} \notin \mathcal{A}P$: indeed, $\id_{\ell^p(G)} - P$ annihilates $\mathcal{A}P$ but not $\id_{\ell^p(G)}$.
				
		\vskip 2mm
		\underline{Case 2: $FC(G)_T = \{e \}.$} By \cite[Theorem 1.6]{T} $FC(G)_T \supset FC(G)'$, so that $FC(G)'$ is also trivial, and $FC(G)$ is abelian. 
		Pick $x \in FC(G)$, and let $N = \langle ccl_G(x) \rangle$, which is a normal subgroup of $G$. Since $ccl_G(x)$ is finite, $N$ is a finitely-generated, torsion-free, abelian group, and hence $N \cong \Z^n$, for some $n \in \N$.
		 As such we may apply Lemma \ref{3.7} to get a weak*-closed ideal $I$ of $CV_p(G)$ with $I \cap \lambda_p(\ell^1(G)) \neq \{ 0 \}$, and hence also $I \cap \mathcal{A} \neq \{ 0 \}$. Moreover, $\id_{\ell^p(G)} \notin I$ implies that $I \cap \mathcal{A}$ is a proper ideal of $\mathcal{A}$, and clearly it is weak*-closed. Hence $\mathcal{A}$ is not weak*-simple.
	\end{proof}

	\section{A Uniform ICC Condition}
	
	In this short section we shall prove a result (Proposition \ref{icc3}) about conjugation in infinite conjugacy classes that we shall require for the proofs of our main results in Section 5. We find it interesting to note that, as a special case, Proposition \ref{icc3} implies that the ICC condition for groups is equivalent to the following uniform version of the condition. 
	
	\begin{definition}	\label{iccdef}
		Let $G$ be an infinite group. We say that $G$ is \textit{uniformly ICC} if for every conjugacy class $Y \neq \{ e \}$, and every pair of finite sets $E, F \subset Y$ there exists $g \in G$ such that $gEg^{-1} \cap F = \emptyset$.
	\end{definition}
	
	
	\begin{lemma} 	\label{icc1}
		Let $G$ be a group, and let $g,h,x \in G$. Then we have
			 $$\{ t \in G: (x^h)^t = x^g\} = gh^{-1}C_G(x)^h.$$
	\end{lemma}
	
	\begin{proof}
		 We have $(x^h)^t = x^g$ if and only if $(g^{-1}th)x(h^{-1}t^{-1}g)  =x$. As such $t$ conjugates $x^h$ to $x^g$ if and only if $g^{-1}th \in C_G(x)$, which is equivalent to
		 $t \in gC_G(x)h^{-1} = gh^{-1}C_G(x)^h$, and the lemma follows.
	\end{proof}
	
\begin{samepage}
	\begin{lemma}		\label{icc2}
		Let $G$ be a group, and let $H$ be a subgroup of $G$. Suppose that, for some $m,n \in \N$, there exist $t_1, \ldots, t_n \in G$, and $g_1, \ldots, g_m \in G$ such that 
		  $$\bigcup_{i = 1}^n \bigcup_{j=1}^m t_iH^{g_j} =G.$$
		 Then $[G:H] <\infty$.
	\end{lemma}
	\end{samepage}

	\begin{proof}
		We shall show that $G$ can be written as the union of finitely many left cosets of the subgroups 
		$H^{g_1}, \ldots, H^{g_{m-1}}$. By repeating this argument $m-1$ times we can then infer that $G$ is the union of finitely many right cosets of $H^{g_1}$, and hence that $[G:H] = [G:H^{g_1}]$ is finite. 
		
		If $t_1, \ldots, t_n$ is a left transversal for $H^{g_m}$ in $G$ then we are done, so suppose that there is some coset $sH^{g_m}$ which is not among $t_1H^{g_m}, \ldots, t_nH^{g_m}$. Then we must have
		$$sH^{g_m} \subset \bigcup_{i = 1}^n \bigcup_{j=1}^{m-1} t_iH^{g_j},$$
		which implies that, for each $k = 1, \ldots, n$, we have
		$$t_kH^{g_m} = (t_ks^{-1})sH^{g_m} \subset \bigcup_{i = 1}^n \bigcup_{j=1}^{m-1} t_ks^{-1}t_iH^{g_j}$$
		and it follows that $G$ is equal to the union of the original cosets of $H^{g_1}, \ldots, H^{g_m}$ together with these new cosets, that is
		$$G = \bigcup_{i = 1}^n \bigcup_{j=1}^{m-1} t_iH^{g_j} \cup 
		\bigcup_{i = 1}^n \bigcup_{k = 1}^n\bigcup_{j=1}^{m-1}  t_ks^{-1}t_iH^{g_j}.$$
		This completes the proof.
	\end{proof}
	
	The following proposition proves, in particular, that being ICC is equivalent to being uniformly ICC. In fact, an even stronger statement is obtained.
	
	\begin{proposition}		\label{icc3}
		Let $G$ be a group with an infinite conjugacy class $Y$, and let $W \subset FC(G)$ be finite. Then for every pair of finite subsets $E, F \subset Y$ there exists $g \in C_G(W)$ such that $gEg^{-1} \cap F = \emptyset$.
	\end{proposition}
	
	\begin{proof}
		Define
		$$D = \{ g \in G : gEg^{-1} \cap F \neq \emptyset \}.$$
		We want to find $g \in C_G(W) \setminus D$.
		If we enumerate the elements of $E$ and $F$ as $E = \{x_1, x_2, \ldots, x_n \}$ and $F = \{ y_1, y_2, \ldots, y_m \}$ then we see that
		\begin{align*}
			D =  
			 \bigcup_{j=1}^m \bigcup_{i = 1}^n \{ g \in G: gx_ig^{-1} = y_j\}.
		\end{align*}
		By Lemma \ref{icc1} $D$ is a finite union of left cosets of conjugates of $C_G(x_1)$.
		
		Assume towards a contradiction that $D \supset C_G(W)$. Since $W \subset FC(G)$ we must have $[G:C_G(w)] < \infty$ for each $w \in W$, and hence $C_G(W) = \bigcap_{w \in W} C_G(w)$ has finite index in $G$. 
		Letting $t_1, \ldots, t_k$ be a left transversal for $C_G(W)$ in $G$, we must have $t_1D \cup  \cdots \cup t_k D = G$. Since $D$ is a finite union of left cosets of conjugates of $C_G(x_1)$, we can now infer that $G$ itself must be a finite union of
		left cosets of conjugates of $C_G(x_1)$. 
		By Lemma \ref{icc2} this implies that $[G : C_G(x_1)] $ is finite, and hence that
		$$|Y| = |G/C_G(x_1)| < \infty,$$
		a contradiction. Therefore $D$ does not contain $C_G(W)$, and taking some $g \in C_G(W) \setminus D$ we have $gEg^{-1} \cap F = \emptyset$, as required.
	\end{proof}

	\section{Weak*-Closed ideals of $\ell^1(G)$}
	
	We begin this section by looking at the $\ell^1$-algebras of infinite abelian groups. The next lemma gives a method to construct many weak*-closed ideals in such algebras.
		
	\begin{lemma}	\label{yc}
		Let $G$ be an infinite abelian group. Take an open set $U \subset \widehat{G}$, with 
		$0< m(\overline{U}) <1$, where $m$ denotes the Haar measure on $\widehat{G}$, and set $E = \overline{U}$. Then
		$$K_E := \{ f \in \ell^1(G) : \widehat{f}(t) = 0 \ (t \in E)\}$$
		is a proper, non-zero, weak*-closed ideal of $\ell^1(G)$.
	\end{lemma}
	
	\begin{proof}
		Note that $I := \{ f \in L^\infty(\widehat{G}) : f(t) = 0, \ \text{a.e } t \in E \}$ is a weak*-closed ideal of $L^\infty(\widehat{G})$, 
		and that the inclusion map $i \colon A(\widehat{G}) \hookrightarrow L^\infty(\widehat{G})$ is weak*-continuous. 
		As such 
		$$i^{-1}(I) = \{ f \in A(\widehat{G}) : f(t) = 0 \ (t \in E) \}$$
		 is weak*-closed in $A(\widehat{G})$, 
		and it is proper and non-zero since $A(\widehat{G})$ is regular.
		Pulling this ideal back to $\ell^1(G)$ via the inverse Fourier transform gives the ideal $K_E$.
	\end{proof}
	
	\begin{remark}		\label{abrem}
	We observe that it follows from {\"U}lger's work \cite{U14} that, for an infinite abelian group $G$, the weak*-closed ideal structure of $\ell^1(G)$ is very complicated, and in particular not all of the weak*-closed ideals of $\ell^1(G)$ are of the form $K_E$. For this discussion we identify $\ell^1(G)$ with $A(\widehat{G})$. In {\"U}lger's notation we have
	$$J(E) : = \overline{\{ u \in A(\widehat{G}) : \supp u \text{ is compact and disjoint from } E\}},$$
	$$k(E) := \{u \in A(\widehat{G}) : u(t) = 0 \ (t \in E)\},$$
	and $k(E)$ is the image of $K_E$ under the Fourier transform.
	{\"U}lger shows in \cite[Example 2.5]{U14} that $\widehat{G}$ contains a closed subset $E$ such that $\operatorname{hull}(\overline{J(E)}^{w^*})= E$, but such that $k(E)$ is not weak*-closed. Hence $\overline{J(E)}^{w^*}$ is a weak*-closed ideal not of the form $k(E)$.
	
	This, together with other results in {\"U}lger's paper such as \cite[Theorem 3.1]{U14}, lead us to suspect that classifying the weak*-closed ideals of $\ell^1(\Z)$ is approximately as hard as classifying the sets of spectral synthesis for $\T$, which is a very old and difficult unsolved problem. 
	\end{remark}
	
	We now move on to looking at general groups. 
	The next lemma allows us to lift certain weak*-closed ideals from a normal subgroup up to the ambient group. 
	The notation $\oplus$ indicates an internal direct sum, and we have identified the subset $\delta_t*J$ of $\ell^1(tN)$ with its canonical image inside $\ell^1(G)$.
	
	\begin{lemma}		\label{0.1}
		Let $G$ be a group, let $N$ be a normal subgroup of $G$, and let $\mathcal{T}$ be a transversal for $N$ in $G$. Let $J$ be a non-trivial weak*-closed ideal of 
		$\ell^1(N)$ which satisfies
		$J^s = J \ (s \in G).$
		Then
		$$I : =  \bigoplus_{t \in \mathcal{T}} \delta_t * J$$
		is a non-trivial weak*-closed ideal of $\ell^1(G)$.
	\end{lemma}

	\begin{proof}
		First, we show that $I$ is weak*-closed. Let $E = J_\perp \subset c_0(N)$, and define 
		$$F =  \bigoplus_{t \in \mathcal{T}} \delta_t *E \leq c_0(G) .$$
		A routine calculation shows that $F^\perp = I$, which implies that $I$ is weak*-closed. 
		
		Next we show that $I$ is an ideal.	Let $f \in I$ and $s \in G$. As $I$ is clearly a linear subspace, it is enough to show that $\delta_s*f, f*\delta_s \in I$.
		Write 
		$f = \sum_{t \in \mathcal{T}} \delta_t*f_t,$
		for some $f_t \in J \ (t \in \mathcal{T}).$ For each $t \in \mathcal{T}$, there exist unique $t', t'' \in \mathcal{T}$, and $n(t), m(t) \in N$, such that
		$$st = t'n(t), \qquad ts =  t'' m(t).$$
		We have
		$$\delta_s *f = \sum_{t \in \mathcal{T}} \delta_{st} * f_t = \sum_{t \in \mathcal{T}} \delta_{t'}*\delta_{n(t)}*f_t.$$
		Since $J \lhd \ell^1(N)$ we must have $\delta_{n(t)} * f_t \in J$ for all $t \in \mathcal{T}$, so that $\delta_s *f$ has the required form and belongs to $I$.
		Similarly
		$$f*\delta_s = \sum_{t \in \mathcal{T}} \delta_{ts} * f_t^{s^{-1}} = \sum_{t \in \mathcal{T}} \delta_{t''}*\delta_{m(t)}*f_t^{s^{-1}}.$$
		By the hypothesis on $J$ we have $\delta_{m(t)}*f_t^{s^{-1}} \in J$ for each $t \in \mathcal{T}$, so that $f*\delta_s$ also has the 
		required form to belong to $I$. 
		
		It is clear that $I$ is non-zero if $J$ is. Moreover, if $\delta_e \in I$ then would have $\delta_e \in J$, so that $I$ is proper. This completes the proof.
	\end{proof}

	Note that, in the previous lemma, the ideal $I$ does not depend on the choice of transversal $\mathcal{T}$, since it can be characterised as the set of those $f \in \ell^1(G)$ with the property that the restriction of $f$ to each coset of $N$ belongs to $J$.
	
	\begin{lemma}  \label{w*lim}
		Let $G$ be a group, and consider the collection of finite subsets of $G$ as a directed set under inclusion. Let $f \in \ell^1(G)$ and $(x_K) \ (K \subset \subset G)$ be a net of elements of $G$ satisfying
		$$(\supp f)^{x_K} \cap K = \emptyset \quad (K \subset \subset G).$$
		Then $f^{x_K} \to 0$ in the weak*-topology.
	\end{lemma}
	
	\begin{proof}
		Let $\phi \in c_0(G)$, and write $F = \supp f$. Then 
		\begin{align*}
			|\langle \phi, f^{x_K} \rangle| &= \left| \sum_{s \in F^{x_K}} f(s) \phi(s) \right| \leq \|f\| \max \{ |\phi(s)| : s \in F^{x_K}\} 
			 \leq \|f\| \max\{ |\phi(s)| : s \in G \setminus K \},
		\end{align*}
		which can be made arbitrarily small for $K$ large enough.
	\end{proof}

	The next lemma is the key to proving that if $G$ is ICC then $\ell^1(G)$ is weak*-simple, as well as our more precise structural result Theorem \ref{thm2}. In the proof, functions restricted to the empty set and empty sums should be interpreted as zero.
	
	\begin{lemma}		\label{icc4}
		Let $G$ be a group and let $I \lhd \ell^1(G)$ be weak*-closed.
		Then for all $f \in I$ we have $f|_{FC(G)} \in I$.
	\end{lemma}

	\begin{proof}
		Let $\eps>0$ and let $A \subset \subset G$ satisfy $\|f|_{G\setminus A} \|< \eps/2$. We wish to find an elements of $I$ that is within $2\eps$ of $f|_{FC(G)}$, and then taking the limit as $\eps \to 0$ will tell us that $f|_{FC(G)} \in I$, as required. If $A \subset FC(G)$ then $f$ itself can be this approximation. Otherwise we proceed as follows.
		
		There are finitely many infinite conjugacy classes represented in $A$ and we shall call them $Y_1, Y_2, \ldots, Y_m$, where $m \in \N$.
		We also define $W = A \cap FC(G)$ and $w = f|_W$.
		Our plan is to define a sequence of functions $f_1, \ldots, f_m \in I$, with the property that each $f_i \ (i = 1, \ldots, m)$ can be decomposed as
		$$f_i = w+ \phi_i + \psi_i,$$
		where 
		\begin{equation}	\label{eq1}
		\supp \phi_i \subset Y_{i+1}\cup \cdots \cup Y_m
		\end{equation}
		and 
		\begin{equation}	\label{eq2}
		\|\psi_i\| \leq \eps/2 + \cdots + \eps/2^i.
		\end{equation}
		It will turn out that the final iteration $f_m$ is within $2\eps$ of $f|_{FC(G)}$, as desired. In order to pass from stage $i$ to stage $i+1$ we shall also have to introduce functions $y_i, g_i,$ and $h_i \ (i = 0, \ldots, m)$, beginning as follows.
		
		Define
		$$y_0 = \sum_{t \in A \cap Y_1} f(t) \delta_t, \qquad g_0 = \sum_{t \in A \setminus(W \cup Y_1)} f(t) \delta_t, 
		\qquad h_0 = \sum_{t \in G \setminus A}f(t) \delta_t,$$
		so that $ f= w + y_0+ g_0 +h_0$ and $\|h_0 \| \leq \eps/2$.
		By Proposition \ref{icc3}, for each $K\subset \subset G$ there exists $x_K \in C_G(W)$ such that $x_K (A \cap Y_1)x_K^{-1}$ has empty intersection with $K \cap Y_1$, from which it follows that
		$$x_K (A \cap Y_1)x_K^{-1} \cap K = \emptyset,$$
		since $x_K (A \cap Y_1)x_K^{-1} \subset Y_1$.
		By  the Banach--Alaoglu Theorem, we may pass to a subset $(x_{K_{\alpha}})$ of $(x_K)$ such that each of the limits
		$$f_1 := \lim_{w^*, \, \alpha} f^{x_{K_{\alpha}}}, \qquad \phi_1 : = \lim_{w^*, \, \alpha} g_0^{x_{K_{\alpha}}}, \qquad \text{and} \qquad \psi_1 : = \lim_{w^*, \, \alpha} h_0^{x_{K_{\alpha}}}$$
		exist.  
		By the choice of $(x_K)$ and Lemma \ref{w*lim} we have $\lim_{w^*, \, \alpha} y_0^{x_{K_\alpha}} = 0$.
		Note also that $w^{x_{K_\alpha}} = w$ for all $\alpha$, since each $x_K$ was chosen to centralise $W$.
		Hence we have 
		$$f_1 = w  + \phi_1 + \psi_1,$$
		and since $I$ is weak*-closed, $f_1 \in I$.
		Moreover, since $g_0^{x_{K_\alpha}}$ is supported on $Y_2 \cup \cdots \cup Y_m$ for all $\alpha$, 
		we have that 
		$$\supp \phi_1 \subset Y_2 \cup \cdots \cup Y_m$$
		as well. Also $\|\psi_1 \| \leq \eps/2 $ because each $\|h_0^{x_K} \| \leq \eps/2 \ (K \subset \subset G)$.
		Hence $f_1$ has a decomposition satisfying \eqref{eq1} and \eqref{eq2}.
		
		Now suppose that $1<i<m$ and that we have
		$$f_i = w+\phi_i+\psi_i \in I,$$
		for $\phi_i$ and $\psi_i$ satisfying \eqref{eq1} and \eqref{eq2} respectively. 
		Let $B \subset \subset G$ satisfy $\|\phi_i|_{G \setminus B} \| < \eps/2^{i+1}$ and define
		$$y_i = \sum_{t \in B \cap Y_{i+1}} \phi_i(t) \delta_t, \qquad g_i = \sum_{t \in B \setminus Y_{i+1}} \phi_i(t) \delta_t, 
		\qquad h_i= \psi_i + \sum_{t \in G \setminus B} \phi_i(t) \delta_t .$$
		Then $f_i = w + y_i + g_i + h_i$, and 
		$$\|h_i \| \leq \|\psi_i \| + \|\phi_i|_{G \setminus B} \| \leq \eps/2 + \cdots + \eps/2^i + \eps/2^{i+1}. $$
		For each $K \subset \subset G$ we may choose $x_K' \in C_G(W)$ such that 
		$$x_K' (B \cap Y_{i+1})x_K'^{-1} \cap K = \emptyset.$$
		Again, by the choice of $(x_K')$, we have $y_i^{x'_K} \to 0$ in the weak*-topology. As before, we may pass to a subnet $(x_{K_\beta}')$
		such that the limits
		$$\phi_{i+1} : = \lim_{w^*, \, \beta} g_i^{x'_{K_\beta}} \qquad \text{and} \qquad \psi_{i+1} := \lim_{w^*, \, \beta} h_i^{x'_{K_\beta}}.$$
		exist, and then
		$$f_{i+1} := \lim_{w^*, \, \beta} f_i^{x_{K_\beta}} = w + \phi_{i+1} + \psi_{i+1}.$$
		We have $\|\psi_{i+1} \| \leq \|h_i\| \leq \eps/2  + \cdots + \eps/2^{i+1}$, so that $\psi_{i+1}$ satisfies \eqref{eq2}. Also, $\phi_{i+1}$ satisfies \eqref{eq1} for similar reasons to $\phi_1$, and $f_{i+1} \in I$ because each $f_i^{x_{K_\beta}} \in I$.

		At the $m^{\rm th}$ stage we find that $\supp \phi_m = \emptyset$, and so $\phi_m$ must be zero. As such $f_m = w+ \psi_m$, and we have
		\begin{align*}
			\|f_m - f|_{FC(G)}\| &\leq \|f_m - w \|+ \|w- f|_{FC(G)} \|  
			= \|\psi_m\| + \|w- (w+h_0|_{FC(G)}) \|  \\
			 &= \|\psi_m \| + \|h_0|_{FC(G)} \| 
			 \leq ( \eps/2 + \cdots +\eps/2^m ) + \eps/2 < 2\eps. 
		\end{align*}
		This completes the proof.
	\end{proof}
	
	We now turn to the proofs of our main results. Firstly, using the previous lemma, we shall characterise weak*-simplicity of $\ell^1(G)$. In the proof $\ell_0^1(N)$ denotes the augmentation ideal of the subgroup $N$, namely $\{ f  \in \ell^1(N) : \sum_{s \in N} f(s) = 0 \}.$
	
	\begin{theorem} 	\label{iccthm}
		Let $G$ be a group. Then $\ell^1(G)$ is weak*-simple if and only if $G$ is ICC.
	\end{theorem}
	
	\begin{proof}
		Suppose that $G$ is an ICC group, and let $\{ 0 \} \neq I \lhd \ell^1(G)$ be weak*-closed. Then, after translating and scaling if necessary, we can find $f \in I$ with $f(e) = 1$. It follows from Lemma \ref{icc4} that $\delta_e \in I$, so that $I = \ell^1(G)$. As $I$ was arbitrary, $\ell^1(G)$ must be weak*-simple.
		
	    Now suppose that $G$ is not ICC so that $FC(G) \neq \{ e \}$. We shall show that $\ell^1(G)$ is not weak*-simple. 		
		There are two cases.
		
		\vskip 2mm
		\underline{Case 1: $FC(G)_T \neq \{e \}.$}
		Pick $x \in FC(G)_T \setminus \{ e \}$. By Dicman's Lemma \cite[Lemma 1.3]{T} there exists a finite normal subgroup $N$ of $G$ with $x \in N$, so that 
		$N \neq \{e \}$. Let $\mathcal{T}$ be a transversal for $N$ in $G$. Then by Lemma \ref{0.1} $I: = \bigoplus_{t \in \mathcal{T}} \delta_t*\ell_0^1(N) $ is a proper, non-trivial, weak*-closed ideal of $\ell^1(G)$.

		\vskip 2mm
		\underline{Case 2: $FC(G)_T = \{e \}.$} 
		Just as in Case 2 of the proof of Theorem \ref{3.8}, we get an $FC$-central, normal subgroup $N \lhd G$, which is isomorphic to $\Z^n$, for some $n \in \N$. Then, as in the proof of Lemma \ref{3.7}, we get a compact neighbourhood of the identity $E$ in $\widehat{N}$, which has $0<m(E)<1$, and which is invariant for the action of $G$ on $\widehat{N}$ induced by conjugation.
		%
		%
		As such, by Lemma \ref{yc}, 
		$$K_E = \{ f \in \ell^1(N) : \widehat{f}|_E = 0 \}$$
		is a non-trivial weak*-closed ideal of $\ell^1(N)$, and it is invariant for the conjugation action of $G$ because $E$ is.
		It follows from Lemma \ref{0.1} that $\ell^1(G)$ has a non-trivial weak*-closed ideal.
	\end{proof}
	
	The next proposition says that the weak*-closed ideals of $\ell^1(G)$ are determined by what happens on $FC(G)$.
	
	\begin{proposition}		\label{idealthm}
		Let $G$ be a group, and let $\mathcal{T}$ be a transversal for $FC(G)$ in $G$.
		Let $I \lhd \ell^1(G)$ be weak*-closed, and let 
		$$J = I \cap \ell^1(FC(G)).$$
		Then $J$ is a weak*-closed ideal of $\ell^1(FC(G))$ invariant for the conjugation action of $G$, and
		\begin{equation}	\label{eq3}
		I = \bigoplus_{t \in \mathcal{T}} \delta_t*J.
		\end{equation}
	\end{proposition}
	
	\begin{proof}
		First observe that $J$ is $G$-conjugation-invariant. Indeed, if $f \in J$ and $t \in G$, then $f^t \in I$ because $I$ is an ideal, but also $f^t \in \ell^1(FC(G))$ as $FC(G)$ is a normal subgroup. 
		
		Clearly the right hand side of \eqref{eq3} is contained in the left hand side, so we shall prove the reverse containment. Indeed, let $f \in I$. Then we can write
		$$f = \sum_{t \in \mathcal{T}} \delta_t*f_t,$$
		where $f_t \in \ell^1(FC(G))$ is given by $(\delta_{t^{-1}}*f)|_{FC(G)}$. For each $t \in \mathcal{T}$ we have $\delta_{t^{-1}}*f \in I$ since $I$ is an ideal, and hence $f_t \in I$ by Lemma \ref{icc4}. Since also $f_t \in \ell^1(FC(G))$ we have $f_t \in J$, and the conclusion follows.
	\end{proof}
	
	We can now prove Theorem \ref{thm2}.
	
	\begin{proof}[Proof of Theorem \ref{thm2}]
		This follows from Proposition \ref{idealthm} and Lemma \ref{0.1}.
	\end{proof}
	
	We now prove our classification theorem for groups with finite $FC$-centres, Theorem \ref{thm3}. Recall that given $\Omega \subset \widehat{FC(G)}$ we define
	$$K(\Omega) : = \bigcap_{\pi \in \Omega} \ker \pi \lhd \ell^1(FC(G)).$$
	For a finite group $G$ the ideals of $\ell^1(G)$ are all given by intersections of kernels of irreducible representations \cite[Theorem 38.7]{HR2}. This fact, together with Theorem \ref{thm2}, implies Theorem \ref{thm3}.
	
	
	
	\begin{proof}[Proof of Theorem \ref{thm3}]
		By Theorem \ref{thm2} it is enough to classify those ideals of $\ell^1(FC(G))$ which are invariant for the conjugation action of $G$; these are automatically weak*-closed since $\ell^1(FC(G))$ is finite-dimensional, by hypothesis. Let $J \lhd \ell^1(FC(G))$ be such an ideal. Then by \cite[Theorem 38.7]{HR2} there exists a subset $\mathcal{O} \subset \widehat{FC(G)}$ such that $J = K(\mathcal{O})$.
		Let $\Omega = \{ t \cdot \sigma : \sigma \in \mathcal{O}, \ t \in G \}.$ Then given $\pi = t \cdot \sigma \in \Omega$ we have $J \subset \ker \sigma$ so $J = J^t \subset \ker (t \cdot \sigma) = \ker \pi$. As such $J \subset \bigcap_{\pi \in \Omega} \ker \pi$, but also clearly $J \supset  \bigcap_{\pi \in \Omega} \ker \pi$. Hence $J = K( \Omega)$, and the result follows.
	\end{proof}

	\subsection*{Acknowledgements}
	I am grateful to Yemon Choi for useful email exchanges, and in particular for pointing out how to construct weak*-closed ideals of $\ell^1(\Z)$. I am also grateful to the referee for their useful comments that have improved the exposition of the article.

\end{document}